\documentclass[12pt,reqno]{amsart}
\usepackage{amsmath}
\usepackage{amsxtra}
\usepackage{amscd}
\usepackage{amsthm}
\usepackage{amsfonts}
\usepackage{amssymb}
\usepackage{eucal}

\def\>{\relax\ifmmode\mskip.666667\thinmuskip\relax\else\kern.111111em\fi}
\def\<{\relax\ifmmode\mskip-.333333\thinmuskip\relax\else\kern-.0555556em\fi}
\def\vsk#1>{\vskip#1\baselineskip}
\def\vv#1>{\vadjust{\vsk#1>}\ignorespaces}
\def\vvn#1>{\vadjust{\nobreak\vsk#1>\nobreak}\ignorespaces}

\let\alb\allowbreak

\let\Maketitle\maketitle
\def\maketitle{\Maketitle\thispagestyle{empty}\let\maketitle\empty}

\newtheorem{theorem}{Theorem}[section]

\newtheorem{cor}[theorem]{Corollary}
\newtheorem{lem}[theorem]{Lemma}
\newtheorem{prop}[theorem]{Proposition}

\theoremstyle{definition}

\theoremstyle{remark}
\newtheorem*{rem}{Remark}
\newtheorem{example}[theorem]{Example}

\numberwithin{equation}{section}

\def\dfrac#1#2{{\displaystyle\frac{#1}{#2}}}

\let\leq\leqslant

\let\geq\geqslant

\def\XXX/{{\sl XXX}}

\let\nc\newcommand
\let\rnc\renewcommand

\nc{\on}{\operatorname}
\nc{\ch}{\mbox{ch}}
\nc{\Z}{{\mathbb Z}}
\nc{\C}{{\mathbb C}}
\nc{\R}{{\mathbb R}}
\nc{\pone}{{\mathbb C}{\mathbb P}^1}
\nc{\pa}{\partial}
\nc{\F}{{\mathcal F}}
\nc{\arr}{\rightarrow}
\nc{\larr}{\longrightarrow}
\nc{\al}{\alpha}
\nc{\ri}{\rangle}
\nc{\lef}{\langle}
\nc{\W}{{\mathcal W}}
\nc{\la}{\lambda}
\nc{\ep}{\epsilon}
\nc{\eps}{\varepsilon}

\nc{\Om}{\Omega}
\nc{\su}{\widehat{{\mathfrak sl}}_2}
\nc{\sw}{{\mathfrak s}{\mathfrak l}}

\nc{\g}{{\mathfrak g}}
\nc{\h}{{\mathfrak h}}
\nc{\n}{{\mathfrak n}}
\nc{\N}{\widehat{\n}}
\nc{\G}{\widehat{\g}}
\nc{\De}{\Delta_+}
\nc{\gt}{\widetilde{\g}}
\nc{\Ga}{\Gamma}
\nc{\one}{{\mathbf 1}}
\nc{\z}{{\mathfrak Z}}
\nc{\zz}{{\mathcal Z}}
\nc{\Hh}{{\mathcal H}_\beta}
\nc{\qp}{q^{\frac{k}{2}}}
\nc{\qm}{q^{-\frac{k}{2}}}
\nc{\La}{\Lambda}
\nc{\wt}{\widetilde}
\nc{\qn}{\frac{[m]_q^2}{[2m]_q}}
\nc{\cri}{_{\on{cr}}}
\nc{\kk}{h^\vee}
\nc{\sun}{\widehat{\sw}_N}
\nc{\hh}{\widehat{\mathfrak h}}
\nc{\HH}{{\mathcal H}_{q,t}}
\nc{\ca}{\wt{{\mathcal A}}_{h,k}(\sw_2)}

\nc{\gl}{\widehat{{\mathfrak g}{\mathfrak l}}_2}
\nc{\el}{\ell}
\nc{\s}{{\mathbf s}}
\nc{\bi}{\bibitem}
\nc{\om}{\omega}
\nc{\WW}{\W_\beta}
\nc{\scr}{{\mathbf S}}
\nc{\ab}{{\mathbf a}}
\nc{\rr}{r}
\nc{\ol}{\overline}
\nc{\con}{qt^{-1} + q^{-1}t}
\nc{\den}{q^{\el-1} t^{-\el+1}+ q^{-\el+1} t^{\el-1}}
\nc{\ds}{\displaystyle}
\nc{\B}{B}
\nc{\A}{{\mathbb A}}
\nc{\GG}{{\mathcal G}}
\nc{\UU}{{\mathcal U}}
\nc{\MM}{{\mathcal M}}
\nc{\CC}{{\mathcal C}}
\nc{\GL}{{}^L G}
\nc{\dzz}{\frac{dz}{z}}
\nc{\Res}{\on{Res}}
\nc{\rep}{{\mathcal R}ep \;}
\nc{\uqg}{U_q \G}
\nc{\uqgg}{U_q \g}
\nc{\Fq}{{\mathbb F}_q}

\nc{\stimes}{\ltimes}
\nc{\K}{\hat{\mathcal K}}
\nc{\Ql}{\ol{\mathbb Q}_\ell}
\rnc{\O}{\hat{\mathcal O}}
\nc{\ga}{\gamma}
\nc{\PL}{{}^L P}
\nc{\E}{\mc E}
\nc{\mc}{\mathcal}
\nc{\mbf}{\mathbf}
\nc{\bb}{{\mathfrak b}}
\nc{\OO}{{\mc O}}
\nc{\Po}{{\mc P}}
\nc{\V}{{\mc V}}
\nc{\yy}{{\mc Y}}
\nc{\M}{\mathcal M}
\nc{\Coh}{{{\mathcal C}oh}}
\nc{\Cohn}{\Coh_n}
\nc{\f}{{\mathcal F}}
\nc{\si}{_E}
\nc{\Gaf}{{\mathbb G}_{a,\Fq}}
\nc{\KK}{{\mathfrak k}}

\nc{\PO}{{\mathbb P^1}}
\nc{\PR}{{\mathbb P^r}}

\nc{\Wr}{{ {\rm Wr}}}
\nc{\gln}{{\mathfrak{gl}_N}}

\def\beq{\begin{equation}}
\def\eeq{\end{equation}}
\def\be{\begin{equation*}}
\def\ee{\end{equation*}}

\nc{\bean}{\begin{eqnarray}}
\nc{\eean}{\end{eqnarray}}
\nc{\bea}{\begin{eqnarray*}}
\nc{\eea}{\end{eqnarray*}}

\nc{\bs}{\boldsymbol}
\nc{\Ref}[1]{{$($\ref{#1}$)$}}

\nc{\U}{\mathcal{U}}

\nc{\sing}{{\rm Sing}\,}

\nc{\Ll}{\La - \al(\bs l)}

\nc{\nL}{L_{\om_r}^{\otimes n}[\mu]}
\nc{\nnL}{L_{\om_r}^{\otimes n}}

\nc{\snL}{ \sing L_{\om_r}^{\otimes n}[\mu]}
\nc{\btz}{ \om(\bs t; \bs z)}
\nc{\SL}{ \mathfrak{sl}}
\nc{\GR}{ {G(r+1,d)}}
\nc{\diag}{{\on{diag}}}
\begin{document}

\hrule width 0pt

\title[On reality property of Wronski maps]
{On reality property of Wronski maps}

\author[E.\,Mukhin, V.\,Tarasov, and A.\,Varchenko]
{E.\,Mukhin$\>^*$, V.\,Tarasov$\>^{\star,*}$, and A.\,Varchenko$\>^\diamond$}

\maketitle

\begin{center}
{\it $^\star\<$Department of Mathematical Sciences\\
Indiana University\,--\>Purdue University Indianapolis\\
402 North Blackford St, Indianapolis, IN 46202-3216, USA\/}

\medskip
{\it $^*\<$St.\,Petersburg Branch of Steklov Mathematical Institute\\
Fontanka 27, St.\,Petersburg, 191023, Russia\/}

\medskip
{\it $^\diamond\<$Department of Mathematics, University of North Carolina
at Chapel Hill\\ Chapel Hill, NC 27599-3250, USA\/}
\end{center}

\medskip
\begin{abstract}
We prove that if all roots of the discrete Wronskian with step 1 of
a set of quasi-exponentials with real bases are real, simple and
differ by at least $1$, then the complex span of this set of
quasi-exponentials has a basis consisting of quasi-exponentials with
real coefficients. This result generalizes the B.~and M.\,Shapiro
conjecture about spaces of polynomials.

The proof is based on the Bethe ansatz method for the \XXX/ model.
\end{abstract}

\maketitle

\section{Introduction} The B.~and M.\,Shapiro conjecture asserts that
if the Wronskian of $N$ polynomials with complex coefficients has real roots
only, then the space spanned by these polynomials has a basis consisting of
polynomials with real coefficients. This conjecture has many algebro-geometric
reformulations and has generated a lot of interest in the past decade,
see for example \cite{RSSS}, \cite{S}.

The B.~and M.\,Shapiro conjecture in the case of two polynomials was proved in
\cite{EG} by complex-analytic methods. In \cite{MTV1} we proved the general
case using a different approach. We showed that a generic space of polynomials
$V$ can be constructed by the Bethe ansatz method for the periodic Gaudin
model. It turns out that the coefficients of the monic differential operator
$D$ of order $N$ annihilating $V$ are eigenvalues of the transfer matrices ---
linear operators, acting on the space of states of the Gaudin model.
If the roots of the Wronskian of $V$ are real, then the transfer matrices are
self-adjoint with respect to a positive definite Hermitian form, hence, their
eigenvalues are real. This implies that the coefficients of the differential
operator $D$ are real, which gives the existence of a basis for $V$ consisting
of polynomials with real coefficients.

In this paper we prove a similar statement about spaces of
quasi-exponentials by the same method. Namely, we prove that if the
Wronskian of $N$ quasi-exponentials $e^{\la_ix}p_i(x)$,
where $\la_i$ are real numbers and $p_i(x)$ are polynomials withs
complex coefficients, has real roots only, then the space spanned by
these quasi-exponentials has a basis such that all polynomials have real
coefficients, see Theorem~\ref{theorem 2}. The proof is based on the
Bethe ansatz for the quasi-periodic Gaudin model. The case
$\la_1=\dots=\la_N=0$ is the original B.~and M.\,Shapiro conjecture.

Using the Bethe ansatz for the quasi-periodic \XXX/ model, we obtain
a similar statement about spaces of quasi-exponentials with the Wronskian
replaced by the discrete Wronskian, see~Theorem~\ref{theorem 1}.
In this case, a new phenomenon occurs: the statement is true only if some
additional restrictions are imposed on the roots of the Wronskian. For example,
it is sufficient to require that the roots of the Wronskian differ by at least
one. The first item of Theorem~\ref{theorem 1} for $N=2$ and $\la_1=\la_2=0$
follows from Theorem~1 in \cite{EGSV}.

We also consider spaces of quasi-polynomials of the form
$x^{z_i}p_i(x,\log x)$, where $z_i$ are real numbers and $p_i(x,y)$
are polynomials with complex coefficients, and their Wronskians.
Theorem~\ref{theorem 3} describes sufficient conditions for such a space
to have a basis such that all polynomials have real coefficients.
Theorem~\ref{theorem 3} is a statement bispectral dual to
Theorem~\ref{theorem 1} in the sense of \cite{MTV5}, \cite{BC}.

Theorems~\ref{theorem 1}, \ref{theorem 2} and \ref{theorem 3}
have reformulations in
terms of explicit matrices depending on two groups of complex parameters,
see Theorems~\ref{theorem 1a}, \ref{theorem 2a}, \ref{theorem 3a}.
For example, if a matrix
\be
\left( \begin{array} {cccccc}
a_1 & \dfrac{1}{\la_2-\la_1} & \dfrac{1}{\la_3-\la_1} &\dots &
\dfrac{1}{\la_N-\la_1}
\\
\dfrac{1}{\la_1-\la_2} & a_2 & \dfrac{1}{\la_3-\la_2} &{} \dots &
\dfrac{1}{\la_N-\la_2}
\\[9pt]
{}\dots & {}\dots & {} \dots & \dots & \dots
\\[6pt]
\dfrac{1}{\la_1-\la_N} & \dfrac{1}{\la_2-\la_N}& \dfrac{1}{\la_3-\la_N}
&{} \dots & a_N
\end{array} \right)
\ee has real eigenvalues and the numbers $\lambda_1,\dots,\lambda_N$ are real,
then the numbers $a_1,\dots, a_N$ are real.
Those reformulations, see Corollaries~\ref{cor 1}, \ref{cor 2}, are
related to properties of Calogero-Moser spaces. They also imply a
criterion for the reality of irreducible representations of Cherednik
algebras, see \cite{HY}. The second proof of this criterion can be
found in \cite{GHY}.

The paper is organized as follows. We state the discrete version of
B.~and M.\,Shapiro conjecture in Section~\ref{state sec} and prove
this result in Section~\ref{proof sec}. In Section~\ref{smooth sec} we
deduce Theorem~\ref{theorem 2} for spaces of quasi-exponentials from
Theorem~\ref{theorem 1}. In Section~\ref{trig sec} we consider
spaces of quasi-polynomials, see Theorem ~\ref{theorem 3}. In
Section \ref{app} we reformulate our results in terms of matrices.

The results of the paper were presented at the workshop ``Contemporary
Schubert Calculus and Schubert Geometry'' at Banff, Canada, in March
of 2007 and the conference "ISLAND III: Algebraic aspects of
integrable systems" in Scotland in July of 2007.

{\bf Acknowledgments.}\enspace
E.M. is supported in part by NSF grant DMS-0601005.
V.T. is supported in part by RFFI grant 05-01-00922.
A.V. is supported in part by NSF grant DMS-0555327.

\section{Spaces of quasi-exponentials and
the discrete Wronski map}\label{state sec}
\subsection{Formulation of the statement}
A function of the form $p(x)Q^{x}$, where $Q$ is a nonzero complex
number with the argument fixed, and $p(x)\in\C[x]$, is called {\it a
quasi-exponential function with base $Q$}.

Fix a natural number $N\geq 2$. Let $\bs Q=(Q_1,\dots,Q_{N})$ be a
sequence of nonzero complex numbers with their arguments fixed. We
always assume that if $Q_i=Q_j$ for some $i,j$ then the chosen
arguments of $Q_i$ and $Q_j$ are the same.

We call a complex vector
space of dimension $N$ spanned by quasi-exponential functions $p_i(x)Q_i^x$,
$i=1,\dots,N$, a {\it space of quasi-exponentials with bases $\bs Q$}.

A quasi-exponential function $p(x)Q^{x}$ is called {\it real}
if $Q\in\R^\times$ and $p(x)\in\R[x]$. The space of quasi-exponentials is
called {\it real\/} if it has a basis consisting of real quasi-exponential
functions.

{\it The discrete Wronskian}
of functions $f_1(x),\dots,f_N(x)$ is the determinant
\beq
\label{dwr}
\Wr^d(f_1,\dots, f_N)\ =\ \det
\left( \begin{array} {cccccc}
f_1(x) & f_1(x+1) & \dots & f_1(x+N-1)
\\
f_2 (x) & f_2(x+1) &{} \dots & f_2(x+N-1)
\\
{}\dots & {} \dots & \dots & \dots
\\
f_N(x) & f_N(x+1) &{} \dots & f_N(x+N-1)
\end{array} \right). \notag
\eeq
The discrete Wronskians of two bases for a vector space of functions
differ by multiplication by a nonzero number.

Let $V$ be a space of quasi-exponentials with bases $\bs Q$.
The discrete Wronskian of any basis for $V$ is a quasi-exponential of the form
$w(x)\prod_{j=1}^N Q_j^x$, where $w(x)\in\C[x]$.
The unique representative with a monic polynomial
$w(x)$ is called {\it the discrete Wronskian of\/} $V$ and is denoted by
$\Wr^d(V)$.

\begin{theorem}\label{theorem 1} Let $V$ be a space of quasi-exponentials
with real bases $\bs Q\in(\R^\times)^N$, and let
$\Wr^d(V)=\prod_{i=1}^n(x-z_i)\prod_{j=1}^NQ_j^x$.
Assume that $z_1,\dots,z_n$ are real. We have:
\begin{enumerate}
\item If $|z_i-z_j|\geq 1$ for all $i\neq j$, then the space $V$ is real.
\item Let $Q_1,\dots,Q_N$ be either all positive or all negative.
Assume that there exists a subset $I\subset\{1,\dots,n\}$ such that
$|z_i-z_j|\geq 1$ for $i\neq j$ provided either $i,j\in I$ or $i,j\not\in I$.
Then the space $V$ is real.
\end{enumerate}
\end{theorem}

Theorem \ref{theorem 1} is proved in Section \ref{proof sec}.

\smallskip
Part (1) of Theorem~\ref{theorem 1} for $N=2$ and $\la_1=\la_2=0$ follows
from Theorem 1 in \cite{EGSV}.

\subsection{Examples}
For $\bs Q\in (\R^\times)^N$ let $\mc L_n(\bs Q)$ be the set
of points $\bs z=(z_1,\dots,z_n)$ of $\R^n$
such that all spaces
of quasi-exponentials with bases $\bs Q$ and the discrete Wronskian
$\Wr^d(V)=\prod_{i=1}^n(x-z_i)\prod_{j=1}^NQ_j^x$ are real.

The inequalities on $z_1,\dots,z_n$ described in Theorem~\ref{theorem
1} give an $n$-dimensional subset of $\mc L_n(\bs Q)$ which does not
depend on $\bs Q$. In examples, these inequalities are sharp, in the
sense that the corresponding hyperplanes are tangent to the boundary
of the set $\bigcup_{\bs Q}\mc L_n(\bs Q)$.

A larger subset of $\mc L_n(\bs Q)$ which depends on $\bs Q$ is described in
Proposition \ref{general cond}. In examples, this subset coincides
with $\mc L_n(\bs Q)$, however, its description is rather ineffective.

\begin{example}
Consider the case $N=2$, $\bs Q=(1,Q)$, $\deg p_1=1$, $\deg p_2=1$.
Then the discrete Wronskian has two zeroes, which we assume to be at $0$
and $A$. This case corresponds to the equation on $a,b$,
\be
\Wr^d(x+a,Q^x(x+b))=Q^x(Q-1)x(x-A)\,,
\ee
which has two solutions:
\begin{align*}
a\, &{}=\,\frac{-QA+A-2Q\pm\sqrt{(Q-1)^2A^2+4Q}}{2(Q-1)}\ , \\
b\, &{}=\,-1-A+\frac{QA-A+2Q\mp\sqrt{(Q-1)^2A^2+4Q}}{2(Q-1)}\ .
\end{align*}
The solutions are real for real $Q,A$ if and only if $A^2\geq -4Q/(Q-1)^2$.
Theorem~\ref{theorem 1} claims that the solutions are real if $A^2\geq 1$,
which gives a sufficient condition because $1\geq-4Q/(Q-1)^2$ for real $Q$.
The condition is sharp because $1=-4Q/(Q-1)^2$ for $Q=-1$.
\end{example}

\begin{example}
\label{example}
Consider the case $N=2$, $\bs Q=(1,1)$, $\deg p_1=1$, $\deg p_2=3$.
Then the discrete Wronskian has two zeroes, which we assume to be at $0,A$
and $B$. This case corresponds to the equation on $a,b,c$,
\be
\Wr^d(x+a,x^3+bx^2+c)=2x(x-A)(x-B)\,,
\ee
which has two solutions:
\begin{align*}
a\, &{}=\, -1/2-(A+B)/3\pm 1/3\sqrt{-AB-3/4+A^2+B^2},\\
b\, &{}=\, -3/2-A-B\pm \sqrt{-4AB-3+4A^2+4B^2}, \\
c\, &{}=\, 1/2+2(A+B)/3+1/3\sqrt{-AB-3+4A^2+4B^2}+AB\,.
\end{align*}
The solutions are real for real $A,B$ if and only if $A^2+B^2-AB-3/4\geq 0$.

The set $A^2+B^2-AB-3/4<0$ in the real plane with coordinates $A,B$
is the interior of an ellipse centered at the origin. This ellipse is inscribed
in the hexagon formed by the lines $|A|=1$, $|B|=1$ and $|A-B|=1$, and
is tangent to the sides of the hexagon at the points $(1,1/2)$, $(-1,-1/2)$,
$(1/2,-1/2)$, $(1/2,1)$, $(-1/2,-1)$, $(-1/2,1/2)$. Theorem~\ref{theorem 1}
claims that the numbers $a,b,c$ are real if the point $(A,B)$ is not inside
the hexagon.
\end{example}

\section{Proof of Theorem \ref{theorem 1}}
\label{proof sec}
\subsection{The discrete Wronskian is a finite algebraic map}
\label{d wr map sec}
Fix a natural number ${N\geq 2}$, natural numbers $n_1,\dots,n_k$ such
that $\sum_{i=1}^kn_i=N$, and a natural number $l$ such that $l>N$.
Fix $\bs Q=(Q_1,\dots,Q_1,\dots,Q_k,\dots,Q_k)\in\C^N$, where
$Q_i\neq 0$, $Q_i\neq Q_j$ if $i\neq j$ and $Q_i$ is repeated $n_i$
times.

Let $\C_l[x]\subset \C[x]$ be the space of all polynomials of degree
less than $l$. For $m\leq l$, let $ Gr(m,l)$ be the Grassmannian of all
$m$-dimensional subspaces in $\C_l[x]$. It is an irreducible projective
complex variety of dimension $m(l-m)$.

Let
\bean\label{n}
n=(l-1)N-\sum_{i=1}^k n_i(n_i-1)+1=lN-\sum_{i=1}^kn_i^2+1.
\eean

Define the {\it discrete Wronski map}:
\be
\Wr^d_{\bs Q}:\ Gr(n_1,l)\times Gr(n_2,l)
\times \dots \times Gr(n_k,l)\to Gr(1,n),
\ee
as follows. For $i=1,\dots, k$, let $V_i\in Gr(n_i,l)$ and let
$p_{i,1}(x),\dots,p_{i,n_i}(x)\in\C[x]$
be a basis for $V_i$. Then the map $\Wr^d_{\bs Q}$ sends the point
$V_1\times\dots\times V_k$
to the line spanned by the polynomial
\be
\Wr^d\Big(p_{1,1}(x)Q_1^x, \dots, p_{1,n_1}(x)
Q_1^x,\dots,p_{k,1}(x)Q_k^x,\dots,p_{k,n_k}(x)Q_k^x\Big)
\prod_{i=1}^kQ_i^{-n_ix}.
\ee

Let $V\in Gr(m,l)$. Then $V$ has a unique basis consisting
of monic polynomials $p_j(x)\in\C[x]$, $j=1,\dots, m$, of the form
\be
p_j(x)=x^{d_{j
}}+\sum_{s=1}^{d_{j}}a_{j,s} x^{d_{j}-s}
\ee
such that $d_r>d_s$ when $r>s$ and
$a_{j,s}=0$ whenever $d_{j}-s=d_{r}$ for some $r$.
We call this basis the {\it standard basis for $V$}.
\begin{prop}\label{d finite}
The discrete Wronski map is a finite algebraic map.
\end{prop}
\begin{proof}
The sets $Gr(n_1,l)\times Gr(n_2,l)\times \dots \times Gr(n_k,l)$
and $Gr(1,n)$ are projective algebraic varieties of dimension $n$.
The discrete Wronski map is a well-defined
algebraic map. We only need to show that every point of $Gr(1,n)$
has a finite number of preimages.

Fix a monic polynomial $w(x)\in\C_n[x]$.

Let $V_1\in Gr(n_1,l),\dots,V_k\in Gr(n_k,l)$ be such that
$\Wr^d_{\bs Q}(V_1\times\dots\times V_k)=\C w(x)$, and let
$p_{i,j}(x)=x^{d_{i,j}}+\sum_{s=1}^{d_{i,j}}a_{i,j,s} x^{d_{i,j}-s}$,
$i=1,\dots,k$, $j=1,\dots,n_i$, be the standard basis for $V_i$. Here
$d_{i,j}$ are non-negative integers such that $d_{i,j}<l$ and
$d_{i,r}>d_{i,s}$ if $r>s$.

We have
\bean\label{equation}
\Wr^d(p_{1,1}(x)Q_1^x,\dots, p_{k,n_k}(x)Q_k^x)=\hspace{125pt}\\ w(x)
\prod_{1\leq
i<j\leq k}(Q_j-Q_i)
\prod_{i=1}^k \big(Q_i^{n_ix}
\prod_{1\leq j<s\leq n_i}(d_{i,s}-d_{i,j})\big).\notag
\eean
Consider equation \Ref{equation}
as a system of algebraic equations on the nontrivial
coefficients
$a_{i,j,s}$ of polynomials $p_{i,j}(x)$. The number of equations equals
the number of variables. We claim that this system has finitely many solutions.

Assume that there exist infinitely many solutions.
Then there exists a curve of solutions $a_{i,j,s}^t$, $t\in \R_+$, such that
some of the coefficients $a_{i,j,s}^t$ tend to infinity in the limit
$t\to\infty$.

Consider the limit $t\to\infty$.
There exist $\al_{1,2,1}^t\in\C$ such that the
main terms of polynomials $p_{1,1}^t$ and
$p_{1,2}^t-\al_{1,2,1}^tp_{1,1}^t$ are linearly independent.
Indeed, let $t^{b_{1,1}}q_{1,1}(x)$ be the main term of $p_{1,1}(x)$ and let
$\deg q_{1,1}=c_{1,1}$. Let $t^{b_{1,2}}q_{1,2}(x)$ be the main term of
$p_{1,2}(x)$. We set $\al_{1,2,1}^t=\beta
a^t_{1,2,d_{1,2}-c_{1,1}}/a^t_{1,1,d_{1,1}-c_{1,1}}$ if
$q_{1,2}(x)=\beta q_{1,1}(x)$ for some $\beta\in\C$
and $\al_{1,2,1}^t=0$ otherwise.

Similarly, there exist $\al_{i,j,r}^t\in\C$ such that for
$i=1,\dots,k$, the main terms as $t\to\infty$ of polynomials
\be \tilde p_{i,j}^t(x)=p_{i,j}^t(x)+\sum_{r=1}^{j-1} \al_{i,j,r}^t
p_{i,r}(x), \ee $j=1,\dots,n_i$, are linearly independent.

If $a_{i,j,s}^t$ tend to infinity and all $a_{i,r,l}$ with $r<j$
remain bounded, then the $s$-th coefficient of $\tilde p_{i,j}^t(x)$
tends to infinity. It follows that $\Wr^d(\tilde
p_{1,1}^t(x)Q_1^x,\dots ,\tilde p_{k,n_k}^t(x)Q_k^x)$ has unbounded
coefficients as $t\to\infty$. But it is equal to
$\Wr^d(p_{1,1}^t(x)Q_1^x,\dots ,p_{k,n_k}^t(x)Q_k^x)$, which does not
depend on $t$. It is a contradiction.

Therefore the number of tuples $V_1\in Gr(n_1,l),\dots,V_k\in
Gr(n_k,l)$ such that $\Wr^d_{\bs Q}(V_1\times\dots\times V_k)=\C
w(x)$, and such that $\deg p_{i,j}(x)=d_{i,j}$, where $p_{i,j},$
$j=1,\dots,n_i$, is the standard basis for $V_i$, is finite for each
choice of $d_{i,j}$. The proposition follows.
\end{proof}

We call a space $V$ of quasi-exponentials with bases $\bs Q$ a {\it
weight zero space} if for any two quasi-exponential functions
$p_1(x)Q^x, p_2(x)Q^x$ in $V$ with the same base and a natural number $j$
such that $\deg p_1(x)<j<\deg p_2(x)$ there exists a quasi-exponential function
$p(x)Q^x$ in $V$ with $\deg p(x)=j$.

\begin{cor}\label{weight zero}
If Theorem \ref{theorem 1} holds for weight zero spaces
of quasi-exponentials, then it holds for all spaces of
quasi-exponentials.
\end{cor}
\begin{proof}
By Proposition \ref{d finite},
it is sufficient to prove Theorem \ref{theorem 1}
for generic values of $z_1,\dots,z_n$.

The set of points $V\in Gr(n,l)$ with the standard basis $p_i(x)$,
and $\deg p_{n+1-i}(x)=l-i$, $i=1,\dots,n$, is dense in
$Gr(n,l)$. Therefore, the set of points in $Gr(n_1,l)\times \dots
\times Gr(n_k,l)$ corresponding to weight zero spaces of
quasi-exponentials is dense. Therefore, Corollary \ref{weight zero}
follows from Proposition \ref{d finite}.
\end{proof}

\subsection{Reduction to the case of generic $\bs Q$}
In this section we show that it is sufficient to prove
Theorem \ref{theorem 1} for the case of generic $\bs Q=(Q_1,\dots,Q_N)$.

Fix some natural number $d>N$. For $i=1,\dots N$, let
$q_i(x)=\sum_{j=0}^{d-1}q_{i,j}x^j$. Set
\bean\label{p=q 1}
p(x,Q,\bs Q)=x^d+\sum_{j=1}^N\prod_{r=1}^{j-1}(Q-Q_r)q_j(x).
\eean

\begin{lem} The function
\beq
\label{WxQ}
W(x,\bs Q)=\frac{ \prod_{i=1}^N Q_i^{-x}}{\prod_{i<j}(Q_j-Q_i) }
\on{Wr}^d \bigl(p(x,Q_1,\bs Q)Q_1^x,\dots, p(x,Q_N,\bs Q)Q_N^x\bigr)
\eeq
is a polynomial in variables $x,Q_1,\dots,Q_N$.
\end{lem}
\begin{proof}
If $Q_i=Q_j$, then $\on{Wr}^d
\big(p(x,Q_1,\bs Q)Q_1^x,\dots, p(x,Q_N,\bs Q)Q_N^x\big)=0$.
Therefore, all denominators cancel.
\end{proof}

Fix natural numbers $n_1,\dots,n_k$, such that $\sum_{i=1}^kn_i=N$.

For $\sum_{j=1}^{s-1} n_j <i\leq \sum_{j=1}^{s} n_j$ we set $m(i)=s$,
$r(i)=i-\sum_{j=1}^{s-1} n_j-1$. Let $\bs Q^0=(Q_1^0,\dots,Q_1^0,
\dots,\alb Q_k^0,\dots,Q_k^0)$, where $Q_i^0$ repeats $n_i$ times. The
$i$-th coordinate of $\bs Q^0$ is $Q_{m(i)}^0$.

We present an explicit set of quasi-exponentials whose
discrete Wronskian equals $W(x,\bs Q^0)$ up to an explicit factor.
Let
\beq\label{p=q 2}
p_i^0(x)=(Q^{-r(i)}(x+Q\partial_Q)^{r(i)}
p(x,Q,\bs Q^0))|_{Q=Q_{m(i)}^0}.
\eeq
Clearly $p_i^0(x)$ is a polynomial in $x$ of degree $d+r(i)$.

\begin{lem} We have
\bean
W(x,\bs Q^0)&= &c(\bs Q^0)
\on{Wr}^d(p_1^0(x)(Q_{m(1)}^0)^x,\dots, p_N^0(x)(Q_{m(N)}^0)^x),\label{expl}\\
c(\bs Q^0)&=&\frac{\prod_{i=1}^k (Q_i^0)^{-n_ix}}{\prod_{1\leq i<j\leq k} (Q_j^0-Q_i^0)^{n_in_j} \prod_{i=1}^k\prod_{j=1}^{n_i-1} (n_i-j)^j}\ .\notag
\eean
\end{lem}
\begin{proof}
For a function $f(Q)$, we call
\be
\tau_{Q,h}^{(1)}f(Q)=\frac{f(Q+h)-f(Q)}{h}
\ee
the {\it discrete derivative of $f$}. The {\it $n$-th discrete
derivative of function $f(Q)$}, is defined recursively
$\tau_{Q,h}^{(n)}f(Q)=\tau_{Q,h}\tau_{Q,h}^{(n-1)}f(Q)$. If $f(Q)$ is a
smooth function, then $\lim_{h\to
0}\tau_{Q,h}^{(n)}f(Q)=f^{(n)}(Q) $, where $f^{(n)}(Q)$ is the
$n$-th derivative of $f(Q)$ with respect to $Q$.

Let
\be
\bs Q^0_h=(Q_1^0,Q_1^0+h,\dots,
Q_1^0+(n_i-1)h,\dots, Q_k^0,Q_k^0+h,\dots,Q_k^0+(n_k-1)h),
\ee
where $h$ is small, and we assume that the argument of $Q^0_i+jh$ continuously
depends on $h$. Since the function $W(x,\bs Q)$ is a polynomial, we can compute
$W(x,\bs Q^0)$ as the limit $\,\lim_{h\to 0}W(x,\bs Q^0_h)$.

Taking suitable linear combinations of rows of the matrix used to compute
the discrete Wronskian in formula \Ref{WxQ} for $W(x,\bs Q^0_h)$, we obtain
the matrix whose $(i,j)$ entry equals
\beq
\label{ijentry}
\bigl(\tau^{(r(i))}_{Q,h}(p(x+j-1,Q,\bs Q^0_h)(Q)^{x+j-1})\bigr)
\big|_{Q=Q^0_{m(i)}+\,r(i)h}
\eeq
and whose determinant equals
\be
\prod_{i=1}^k h^{n_i(n_i-1)/2}\;
\on{Wr}^d \bigl(p(x,Q_1,\bs Q)Q_1^x,\dots, p(x,Q_N,\bs Q)Q_N^x\bigr)
\big|_{\bs Q=\bs Q^0_h}\,.
\ee
Comparing expression \Ref{ijentry} with the right hand side of \Ref{p=q 2},
we get formula \Ref{expl} from formula \Ref{WxQ} in the limit $h\to 0$.
\end{proof}

\begin{prop}
Assume that Theorem \ref{theorem 1} holds for generic values of $\bs Q$.
Then Theorem \ref{theorem 1} holds for all $\bs Q\in(\R^\times)^N$.
\end{prop}
\begin{proof}
Let $z_1,\dots,z_n$ be real, satisfying one of the conditions in
Theorem \ref{theorem 1}. Consider the equation $W(x,\bs
Q)=\prod_{s=1}^n(x-z_s)$ as a system of equations on variables
$q_{i,j}$ depending on
parameters $Q_1,\dots,Q_N$. It is a system of algebraic equations with
polynomial dependence on parameters. By the assumption all
solutions of the system for generic real values of $Q_1,\dots,Q_N$ are real.
Therefore, for all real $Q_1,\dots,Q_N$, all solutions of the system are real.
This proves Theorem \ref{theorem 1}
for the weight zero spaces of quasi-exponentials.
Then the proposition follows from Corollary \ref{weight zero}.
\end{proof}

\subsection{Bethe algebra}
In this section we recall some results of \cite{MTV2}, \cite{MTV3}.

Let $W=\C^N$ with a chosen basis $v_1,\dots,v_N$.

For an operator $M\in\on{End} W$, we denote $M^{(i)}=1^{\otimes(
i-1)}\otimes M\otimes 1^{\otimes(n-i)}$. Similarly, for an operator
$M\in\on{End}(W^{\otimes 2})$, we denote by $M^{(ij)}\in
\on{End}(W^{\otimes n})$ the operator acting as $M$ on the $i$-th and
$j$-th factors of $W^{\otimes n}$.

Let $R(x)=x+P \in\on{End}(W^{\otimes 2})$ be the {\it
rational $R$-matrix}. Here $P \in\on{End}(W^{\otimes 2})$ is the
flip map: $P(x\otimes y)=y\otimes x$ for all $x,y\in W$.
Let $E_{ab}\in\on{End} W$ be the
linear operator with the matrix $(\delta_{ia}\delta_{jb})_{i,j=1}^N$.

Let the {\it Yangian} $Y(\gln)$ be the complex unital associative algebra
with generators
$T_{ab}^{\{s\}}$, $a,b=1,\dots,N$, $s\in\Z_{\geq 1}$, and relations
\bean\label{rel}
R^{(12)}(x-y)T^{(13)}(x)T^{(23)}(y)=T^{(23)}(y)T^{(13)}(x)R^{(12)}(x-y),
\eean
where $T(x)=\sum_{a,b=1}^N E_{ab}\otimes T_{ab}(x)$ and
$T_{ab}(x)=\delta_{ab}+\sum_{s=1}^\infty T_{ab}^{\{s\}}x^{-s}$.

The Yangian $Y(\gln)$ is a Hopf algebra, and the coproduct is given by
\be
\Delta(T_{ab}(x))=\sum_{i=1}^NT_{ib}(x)\otimes T_{ai}(x).
\ee
The Yangian $Y(\gln)$ is a flat deformation of $U\gln[t]$, the universal
enveloping algebra of the current algebra $\gln[t]$.

Given $z\in \C$, define the $Y(\gln)$-module structure on the
space $W$ by letting
$T_{ab}(x)$ act as $E_{ba}/(x-z)$. We denote this module $W(z)$ and call it
the {\it evaluation module}.

\medskip
For a matrix $M=(M_{ij})$
with possibly noncommuting entries, we define the row determinant by
$
\on{rdet}(M)=\sum_{\sigma\in S_N} M_{1\sigma(1)}M_{2\sigma(2)}
\dots M_{N\sigma(N)}.
$

Let $\bs Q=(Q_1,\dots,Q_N)\in(\C^\times)^N$. Let
$Q=\on{diag}(Q_1,\dots,Q_N)$ be the diagonal matrix with diagonal
entries $Q_i$ .
Let $\partial=\partial/\partial x$. Define the universal difference operator by
\be
\mc D_{\bs Q}=\on{rdet}(1-QT(x)e^{-\partial}).
\ee
Write
\be
\mc D_{\bs Q}=1-B_{1,\bs Q}(x)e^{-\partial}+B_{2,\bs Q}(x)e^{-2\partial}-
\dots +(-1)^N B_{N,\bs Q}(x)e^{-N\partial}.
\ee
Then $B_{i,\bs Q}(x)$ are series in $x^{-1}$ with
coefficients in $Y(\gln)$. The series $B_i(x)$
coincides with the higher transfer-matrices, see \cite{CT}, \cite{MTV2}.

We call the unital subalgebra of $Y(\gln)$
generated by the
coefficients of the series $B_{i,\bs Q}(x)$, $i=1,\dots,N$, the {\it Bethe
algebra} and denote it by $\mc B_{\bs Q}$.
It is known that the Bethe algebra is commutative, see \cite{KS}.

\medskip

Let $\bs z=(z_1,\dots,z_n)\in\C^n$. Let
$\bs W(\bs z)=W(z_1)\otimes \dots\otimes W(z_n)$ be
the tensor product of the evaluation modules.

Let $\bar B_{i,\bs Q}(x)$, $i=1,\dots,N,$ be the image of $B_{i,\bs Q}(x)$
in $(\on{End }\bs W(\bs z))[[x^{-1}]]$. The series $\bar B_{i,Q}(x)$
is summed up to a rational function in $x$.

Let $K_i$, $i=1,\dots,n$, be the {\it $qKZ$ Hamiltonians} in $\bs W(\bs z)$:
\bean\label{qkz op}
K_i= R^{(i,i-1)}(z_i-z_{i-1})\dots R^{(i,1)}(z_i-z_{1})
Q^{(i)}R^{(i,n)}(z_i-z_n)\dots R^{(i,i+1)}(z_i-z_{i+1}).
\eean

\begin{lem}\label{qKZ} For $i=1,\dots,n$, we have
\be
K_i=\prod_{j,j\neq i}(z_i-z_j)\on{Res}_{x=z_i} \bar B_{1,\bs Q}(x).
\ee
In particular, $K_i$ belongs to the image of $\mc B_{\bs Q}$ in
$\on{End }\bs W(\bs z)$.
\end{lem}
\begin{proof}
The formula is proved by a direct computation.
\end{proof}

The space $W^{\otimes n}$ has the standard {\em tensor Shapovalov form}:
\be
\langle v_{a_1}\otimes\dots\otimes v_{a_n},
v_{b_1}\otimes\dots\otimes v_{b_n}\rangle=\prod_{i=1}^n\delta_{a_ib_i}.
\ee
Recall that the module $\bs W(\bs z)$ as a vector space
is identified with $W^{\otimes n}$.
Let $\langle \ ,\ \rangle_{\bs R}$ be the form on $\bs W(\bs z)$ defined by
\bea
\langle v,w\rangle_{\bs R}&=&\langle v, \bs R w\rangle,\\
\bs R &=& R^{(n-1,n)}(z_{n-1}-z_n)\dots R^{(2,n)}(z_2-z_n) \dots
R^{(2,3)}(z_2-z_3)\times \\
&&R^{(1,n)}(z_1-z_n)\dots R^{(1,3)}(z_1-z_3) R^{(1,2)}(z_1-z_2).
\eea
We call this form the {\it Yangian form} .

\begin{lem}\label{cite1}(\cite{MTV2})
For any $b\in \mc B_{\bs Q}$, $v,w\in\bs W(\bs z)$ we have
\be
\langle bv,w\rangle_{\bs R}=\langle v,bw\rangle_{\bs R}.
\vv->
\ee
\qed
\end{lem}

Let $v\in \bs W(\bs z)$ be an eigenvector of the Bethe algebra
$\mc B_{\bs Q}$. For $i=1,\dots,N$, let
$B_{i,\bs Q,v}(x)$ be the rational function in $x$
with complex coefficients such that
\be
B_{i,\bs Q}(x)v=B_{i,\bs Q,v}(x)v.
\ee
We denote by $\mc D_{\bs Q,v}$ the scalar difference operator
\be
\mc D_{\bs Q,v}=1-B_{1,\bs Q,v}(x)e^{-\partial}+
B_{2,\bs Q,v}(x)e^{-2\partial}-
\dots +(-1)^N B_{N,\bs Q,v}(x)e^{-N\partial}.
\ee

Given a scalar difference operator $\mc D$,
we call the space of all solutions $f(x)$
of the equation $\mc D f(x)=0$ such that $f(x)$
is a linear combination of quasi-exponential functions
the {\it quasi-exponential kernel of operator $\mc D$}.

Let $\mc U$ be the complex span of $1$-periodic quasi-exponentials
$e^{2\pi \sqrt{-1}kx}$, $k\in\Z$.

\begin{lem}\label{cite2}(\cite{MTV3})
Let $v\in \bs W(\bs z)$ be an eigenvector of the Bethe algebra $\mc
B_{\bs Q}$. Then the quasi-exponential kernel of the operator $\mc
D_{\bs Q,v}$ has the form $V_v\otimes \mc U$, where $V_v$ is an
$N$-dimensional complex space of quasi-exponentials with bases $\bs
Q$, and the discrete Wronskian
$\Wr^d(V_v)=\prod_{i=1}^n(x-z_i)\prod_{j=1}^N Q_j^x$.

Moreover, for generic $\bs z,\bs Q$, and every $N$-dimensional
complex space $V$ of quasi-exponentials with bases $\bs Q$ and
$\Wr^d(V)=\prod_{i=1}^n(x-z_i)\prod_{j=1}^N Q_j^x$, there exists an
eigenvector $v\in \bs W(\bs z)$ of the Bethe algebra $\mc B_{\bs
Q}$ such that the quasi-exponential kernel of the operator $\mc
D_{\bs Q,v}$ has the form $V\otimes \mc U$. \qed
\end{lem}

\subsection{Proof of Theorem \ref{theorem 1} for
the case of generic $Q_1,\dots,Q_N$}\label{end of proof}

Let all $z_1,\dots,z_n$ be real.
Let all $Q_1,\dots,Q_N$ also be real and nonzero.

Let $W^\R$ be the real part of $W$ generated by the chosen basis
$v_1,\dots, v_N$, and let $\bs W^\R(\bs z)=W^\R(z_1)\otimes
\dots\otimes W^\R(z_n)$ be the real part of $\bs W(\bs z)$.
Let $Y^\R(\gln)$ be the real unital algebra generated by $T_{a,b}^{\{s\}}$,
$a,b=1,\dots,N$, $s\in\Z_{\geq 1}$, and relations \Ref{rel}.
Let $\mc B_{\bs Q}^\R\subset Y^\R(\gln)$ be the real subalgebra generated
by the coefficients of the series $B_{i,\bs Q}(x)$, $i=1,\dots,N$.
Clearly, $\mc B_{\bs Q}^\R$ acts in the space $\bs W^\R(\bs z)$.

For $g\in \mc B_{\bs Q}^\R$, define the form $\langle\ ,\ \rangle_{\bs R g}$ on
$\bs W^\R(\bs z)$ by the formula
\be \langle v,w\rangle_{\bs R g}=\langle v,gw\rangle_{\bs R}=
\langle v,\bs R g w\rangle.
\ee
The form $\langle\ ,\ \rangle_{\bs R g}$ is a real bilinear symmetric form.

\begin{prop}\label{general cond}
Let $z_1,\dots,z_n$ be real numbers. Let $g\in \mc
B_{\bs Q}^\R$ be such that the form $\langle\ ,\ \rangle_{\bs R g}$ is positive
definite on $\bs W^\R(\bs z)$. Let\/ $V$ be a space of
quasi-exponentials with bases $\bs Q$ and
$\Wr^d(V)=\prod_{i=1}^n(x-z_i)\prod_{j=1}^N Q_j^x$. Then $V$ is
real.
\end{prop}
\begin{proof}
Since the condition of a form being positive definite is open,
we can assume that $\bs z$, $\bs Q$ are generic.
Then by Lemma \ref{cite2}, there exists a vector $v\in\bs W(\bs z)$,
such that $v$ is an eigenvector of the Bethe algebra $\mc B_{\bs Q}$, and
$V\otimes \mc U$
is the quasi-exponential kernel of the operator $\mc D_{\bs Q,v}$.
By Lemma \ref{cite1}, the coefficients of
the operator $\mc D_{\bs Q}$ are rational functions in $x$ which are
symmetric operators with respect to the form $\langle\ ,\ \rangle_{\bs R g}$.
Since this form is positive definite on $\bs W^\R(\bs z)$, the coefficients of
the operator $\mc D_{\bs Q,v}$ are rational functions with real coefficients.

Let $Q$ be a real number. Consider the equation $\mc D_{\bs
Q,v}p(x)Q^x=0$ as a system of equations for the coefficients of the
polynomial $p(x)=\sum_{i=0}^na_{n-i}x^i$. This is a system of linear
equations with real coefficients. Therefore, the space of solutions
has a real basis. The proposition follows.
\end{proof}

In Example \ref{example}, the converse to Proposition \ref{general
cond} is also true. Namely, let $Q_1,\dots,Q_N$ be real. Let
$z_1,\dots,z_n$ be real numbers such that every space of
quasi-exponentials $V$ with bases $\bs Q$ and
$\Wr^d(V)=\prod_{i=1}^n(x-z_i)\prod_{j=1}^N Q_j^x$, is real. Then
there exists $g\in \mc B_{\bs Q}$ such that the form $\langle\ ,\
\rangle_{\bs R g}$ is positive definite on $\bs W^\R(\bs z)$. However,
the existence of such $g\in \mc B_{\bs Q}$ is usually difficult to
check.

\medskip

We deduce
Theorem \ref{theorem 1} from Proposition \ref{general cond}.

\begin{lem}\label{positive}
Assume that $z_i-z_j>1$ if $i>j$.
Then the restriction of the Yangian form to $\bs W^{\R}(\bs z)$
is a positive definite bilinear form.
\end{lem}
\begin{proof}
The restriction of the tensor
Shapovalov form to $(W^{\R})^{\otimes n}$
is a positive definite bilinear form. In the limit
$z_1\gg z_2\gg \dots \gg z_n$, the Yangian form on $\bs W(\bs z)$
tends to the tensor
Shapovalov form. Moreover, the Yangian form is nondegenerate if
$z_i-z_j>1$ for all $i>j$. The lemma follows, since the
dependence of the Yangian form on $\bs z$ is continuous.
\end{proof}

The first part of Theorem \ref{theorem 1} with the additional
condition $z_i-z_j\neq 1$ for all $i,j$ follows from Lemma
\ref{positive} and Proposition \ref{general cond} with $g=1$. Then the
condition that $z_i-z_j\neq 1$ for all $i,j$ can be dropped by the
continuity with respect to $z_i$, see Proposition \ref{d finite}.

\medskip

Assume that $Q_1,\dots,Q_N$ are all positive.
Assume there exists $0\leq s\leq n$
such that $z_i-z_j>1$ if
either $s\geq i>j\geq 1$ or $n\geq i>j>s$.
Consider $G_s=(K_1K_2\dots K_s)^{-1}$, where
$K_i$ are given by \Ref{qkz op}.

If $\bs z$ is generic, then there exists an element
$g_s\in\mc B_{\bs Q}$ which acts on $\bs W(\bs z)$ by $G_s$. Indeed,
$K_i\in\on{End}(\bs W(\bs z))$ are in the image of the Bethe algebra
by Lemma \ref{qKZ} and the inverse of a nondegenerate operator in
a finite-dimensional space can be written as a polynomial of the operator
itself.

\begin{lem}\label{positive2} Assume that $Q_1,\dots,Q_N$ are all positive.
Assume there exists $0\leq s\leq n$
such that $z_i-z_j>1$ if
either $s\geq i>j\geq 1$ or $n\geq i>j>s$. Then
the form $\langle\ ,\ \rangle_{\bs R G_s}$ is a
positive definite bilinear form on $\bs W^\R(\bs z)$.
\end{lem}
\begin{proof}
We have
\bea
\bs R G_s&=&
(R^{(s-1,s)}\dots R^{(2,s)} \dots
R^{(2,3)}R^{(1s)}\dots R^{(1,3)} R^{(1,2)})\times \\
&&(R^{(n-1,n)}\dots R^{(s+2,n)} \dots
R^{(s+2,s+3)}R^{(s+1,n)}\dots R^{(s+1,s+3)} R^{(s+1,s+2)})
(Q^{(1)}\dots Q^{(s)})^{-1},
\eea
where $R^{(i,j)}=R^{(i,j)}(z_i-z_j)$.
In the limit $z_1\gg z_2\gg \dots \gg z_n$,
the form $\langle\ ,\ \rangle_{\bs R G_s}$
tends to the positive definite form $\langle\ ,\ \rangle_{s}$ given by
\be
\langle v, w\rangle_{s}=\langle v, (Q^{(1)}\dots Q^{(s)})^{-1}W\rangle.
\ee
Moreover, the form $\langle\ ,\ \rangle_{\bs R G_s}$
is clearly nondegenerate if
$z_i-z_j>1$ for all $i>j$ such that either $i\leq s$ or $j>s$.
The lemma follows since the
dependence of the form $\langle\ ,\ \rangle_{\bs R G_s}$
on $\bs z$ is continuous.
\end{proof}

The second part of
Theorem \ref{theorem 1} with positive $Q_1,\dots,Q_N$, and the additional
condition that $z_i-z_j\neq 1$ for all $i,j$,
follows from Lemma \ref{positive2} and
Proposition \ref{general cond} with $g=G_s$. Then the condition
that $z_i-z_j\neq 1$ for all $i,j$ can be dropped by the continuity
with respect to $z_i$, see Proposition \ref{d finite}.

\section{Spaces of quasi-exponentials and the differential Wronski map}
\label{smooth sec}

\subsection{Formulation of statement}
\label{smooth statement section}
A function of the form
$p(x)e^{\la x}$, where $\la\in\C$
and $p(x)\in\C[x]$, is called {\it a quasi-exponential function with
exponent $\la$}.

Fix a natural number $N\geq 2$. Let $\bs \la=(\la_1,\dots,\la_N)\in\C^N$.
We call a complex vector space of dimension $N$ spanned by
$N$ quasi-exponential functions $p_i(x)e^{\la_i x}$, $i=1,\dots,N$,
a {\it space of quasi-exponentials with exponents $\bs \la$}.

A quasi-exponential function $p(x)e^{\la x}$
is called {\it real} if $\la\in\R$ and $p(x)\in\R[x]$.
The space of quasi-exponentials $V$ is called {\it real} if it has a basis
consisting of real quasi-exponential functions.

{\it The Wronskian} of functions $f_1(x),\dots,f_N(x)$
is the determinant
\bean\label{wr}
\Wr (f_1,\dots, f_N)\ =\ \det
\left( \begin{array} {cccccc}
f_1 & f_1' & \dots & f_1^{(N-1)}
\\
f_2 & f_2' &{} \dots & f_2^{(N-1)}
\\
{}\dots & {} \dots & \dots & \dots
\\
f_N & f_N' &{} \dots & f_N^{(N-1)}
\end{array} \right).
\eean
The Wronskians of two bases for a vector space of functions differ by
multiplication by a nonzero number.

Let $V$ be a space of quasi-exponentials with exponents $\bs \la$.
The Wronskian of any basis for $V$ is a quasi-exponential of the form
$w(x)e^{\sum_{i=1}^N\la_i x}$, where $w(x)\in\C[x]$.
The unique representative with a monic
polynomial $w(x)$ is called {\it the Wronskian of\/} $V$ and is denoted
by $\Wr(V)$.

\begin{theorem}\label{theorem 2}
Let $V$ be a space of quasi-exponentials with real exponents
$\bs\la\in\R^N$. If
zeroes of the Wronskian $\Wr(V)$ are real, then the space $V$ is real.
\end{theorem}

Theorem \ref{theorem 2} is proved in Section \ref{smooth proof sec}.

Theorem \ref{theorem 2} in the case $\la_1=\la_2=\dots=\la_N=0$ is
the B.~and M.\,Shapiro conjecture proved in \cite{EG} for $N=2$ and
in \cite{MTV1} for all $N$.

\subsection{Proof of Theorem \ref{theorem 2}}\label{smooth proof sec}
Theorem \ref{theorem 2} can be proved similarly to Theorem
\ref{theorem 1}. However, it is not difficult to deduce Theorem
\ref{theorem 2} from Theorem \ref{theorem 1}; we do that in this section.

Let $\bs \la=(\la_1,\dots,\la_1,\dots,\la_k,\dots,\la_k)$,
where $\la_i$ is repeated $n_i$ times.
Consider the {\it Wronski map}:
\be
\Wr_{\bs\la}:\ Gr(n_1,l)\times Gr(n_2,l)
\times \dots \times Gr(n_k,l)\to Gr(1,n)
\ee
which maps
$V_1\times\dots\times V_k$ to
\be
\langle \Wr\Big(p_{1,1}(x)e^{\la_1 x}, \dots,
p_{1,n_1}(x)e^{\la_1 x},\dots , p_{k,1}(x)
e^{\la_k x},\dots,p_{k,n_k}(x)e^{\la_k x}
\Big)\prod_{i=1}^k e^{- n_i\la_ix}\rangle,
\ee
where $n$ is given by \Ref{n}, and
we used the notation of Section \ref{d wr map sec} for the bases for $V_i$.

\begin{prop}\label{finite} The Wronski map is a finite algebraic map.
\end{prop}
\begin{proof}
The proof of Proposition \ref{finite} is similar to
the proof of Proposition \ref{d finite}.
\end{proof}

For $h\in\C^*$,
{\it the discrete Wronskian with step $h$}
of functions $f_1(x),\dots,f_N(x)$ is the determinant
\bean
\label{dwrh}
\Wr^d_h(f_1,\dots, f_N)\ =\ \det
\left( \begin{array} {cccccc}
f_1(x) & f_1(x+h) & \dots & f_1(x+h(N-1))
\\
f_2 (x) & f_2(x+h) &{} \dots & f_2(x+h(N-1))
\\
{}\dots & {} \dots & \dots & \dots
\\
f_N(x) & f_N(x+h) &{} \dots & f_N(x+h(N-1))
\end{array} \right).
\eean

The discrete Wronskians with step $h$
of any two bases for a vector space of functions
differ by multiplication by a nonzero number.
Let $V$ be a space of quasi-exponentials with exponents $\bs\la$.
The discrete Wronskian with step $h$
of any basis for $V$ is a quasi-exponential of the form
$w(x)\prod_{j=1}^Ne^{\la_jx}$, where $w(x)\in\C[x]$.
The unique representative with a monic polynomial
$w(x)$ is called {\it the discrete Wronskian of $V$ with step $h$}
and is denoted by $\Wr^d_h(V)$.

Theorem \ref{theorem 1} implies the following statement.

\begin{cor}\label{cor theorem 1} Let $h$ be real.
Let $V$ be a space of quasi-exponentials
with real exponents $\bs \la\in(\R^\times)^N$, and let
$\Wr^d_h(V)=\prod_{i=1}^n(x-z_i)\prod_{j=1}^Ne^{\la_jx}$.
Assume that $z_1,\dots,z_n$ are real and $|z_i-z_j|\geq |h|$ for all $i\neq j$.
Then the space $V$ is real.
\end{cor}
\begin{proof}
Let $V$ be a space of quasi-exponentials
with real exponents $\bs \la\in(\R^\times)^N$, and let
$\Wr^d_h(V)=\prod_{i=1}^n(x-z_i)\prod_{j=1}^Ne^{\la_jx}$. Then
\be
\bar V=\{f(xh)\ |\ f(x)\in V\}
\ee
is a space of quasi-exponentials with real bases
$(e^{h\la_1},\dots,e^{h\la_N})$, and
\be
\Wr^d(\bar V)=\prod_{i=1}^n(x-z_i/h)\prod_{j=1}^N e^{h\la_jx}\;.
\ee
Therefore the corollary follows from Theorem \ref{theorem 1}.
\end{proof}

Define the {\it discrete Wronski map with step $h$}:
\be
\Wr^d_{\bs \la,h}:\ Gr(n_1,l)\times Gr(n_2,l)
\times \dots \times Gr(n_k,l)\to Gr(1,n)
\ee
which maps $V_1\times\dots\times V_k$ to
\be
\langle \Wr^d_h\Big(p_{1,1}e^{\la_1x}, \dots, p_{1,n_1}(x)
e^{\la_1x},\dots, p_{k,1}(x)e^{\la_kx},\dots, p_{k,n_k}(x)e^{\la_kx}\Big)
\prod_{i=1}^ke^{-n_i\la_ix}\rangle,
\ee
where we used the notation of Section \ref{d wr map sec}
for the bases for $V_i$.

Let
\be
\bar {\Wr}^d_{\bs\la}:\ \C\times Gr(n_1,l)\times Gr(n_2,l)
\times \dots \times Gr(n_k,l)\to \C\times Gr(1,n)
\ee
be the map which equals
\bea
id\times \Wr^d_{{\bs\la},h}\quad
&\on{on}&\quad \{h\}\times Gr(n_1,l)\times \dots \times Gr(n_k,l),\quad
h\in\C^\times,\\
id\times \Wr_{\bs\la} \hspace{17pt} &\on{on}& \quad
\{0\}\times Gr(n_1,l)\times \dots \times Gr(n_k,l).
\eea
\begin{lem}\label{smooth}
The map $\bar{\Wr}^d_{\bs\la}$ is a continuous map of smooth varieties.
\end{lem}
\begin{proof}
Taking linear combinations of the columns in the matrix used to compute the
determinant \Ref{dwrh}, we obtain the
matrix of the discrete derivatives which tend to the usual
derivatives as $h\to 0$.
In particular, let $p_1(x),\dots,p_N(x)\in\C[x]$
and $\bs\la\in \C^N$.
Then the function
$\Wr^d_h((p_1(x)e^{\la_1x},\dots,p_N(x)e^{\la_Nx})$ is
a smooth function of $h$ and
\be
\Wr^d_h(e^{\la_1x}p_1(x),\dots,e^{\la_Nx}p_N(x))=
h^{N(N-1)/2}\Wr(p_1(x)e^{\la_1x},\dots,p_N(x)e^{\la_Nx})+o(h^{N(N-1)/2})
\ee
as $h\to 0$, see \Ref{wr}, \Ref{dwrh}. The lemma follows.
\end{proof}

{}From Proposition \ref{finite}, we obtain that
it is sufficient to prove Theorem \ref{theorem 2}
for the case of distinct zeroes of the Wronskian $Wr(V)$.

Let $w(x)=\prod_{i=1}^n(x-z_i)$ and $z_i\neq z_j$, $i\neq j$. Let $V$
be a space of quasi-exponentials with exponents $\bs\la=(\la_1,\dots,\la_N)$
such that $\Wr(V)=w(x)e^{\sum_{i=1}^N\la_ix}$.
Then by Lemma \ref{smooth}, there exists a
family $V_h$ of spaces of quasi-exponentials with exponents
$\bs\la$ such that
$\Wr_h^d(V_h)=\Wr(V)$ and
$V_h\to V$ as $h\to 0$.
Then by Corollary \ref{cor theorem 1},
for real $h$ such that $|h|\leq \min_{1\leq i<j\leq n}|z_i-z_j|$,
the spaces $V_h$ are real.
It follows that the space $V$ is real.

Theorem \ref{theorem 2} is proved.

\section{Unramified spaces of
quasi-polynomials and the differential Wronski map}\label{trig sec}

\subsection{Formulation of the statement}
A function of the form $p(x,\log x)x^z$, where $z$ is a complex number,
and $p(x,y)\in\C[x,y]$, is called {\it a quasi-polynomial
function with exponent $z$}.

The quasi-polynomials are multi-valued functions and the exponents are
defined modulo integers. This does not present any difficulty, since
in this paper we use only algebraic properties of the quasi-polynomial
functions.

Fix a natural number $n\geq 2$. Let $\bs z=(z_1,\dots,z_{n})$ be a
sequence of complex numbers. We call a complex vector space of
dimension $n$ spanned by quasi-polynomial functions $p_i(x,\log
x)x^{z_i}$, $i=1,\dots,n$, a {\it space of quasi-polynomials with
exponents $\bs z$}.

A quasi-polynomial function $p(x,\log x)x^{z}$ is called {\it real}
if $z\in\R$ and $p(x,y)\in\R[x,y]$. The space of quasi-polynomials is
called {\it real\/} if it has a basis consisting of real quasi-polynomial
functions.

\medskip

The space of quasi-polynomials $V$ is called {\it non-degenerate} if it
does not contain monomials of the form $x^z$.

Given a space of quasi-polynomials $V$, let $\mc D_V=x^n\partial ^n+\dots$
be the unique
differential operator of order $n$ with kernel $V$ and top coefficient $x^n$.
The space of quasi-polynomials $V$ is called {\it unramified}
if coefficients of $\mc D_V$ are rational functions of $x$.

Let $V$ be an $n$-dimensional non-degenerate unramified space of
quasi-polynomials with exponents $\bs z$.

The operator $\mc D_V$ is Fuchsian. Let $\chi_V^{(\infty)}(\al)$ and
$\chi_V^{(0)}(\al)$ be the indicial polynomials of $\mc D_V$ at $x=\infty$
and $x=0$ respectively.
The polynomials $\chi_V^{(\infty)}(\al), \chi_V^{(0)}(\al)$ are
polynomials in $\al$ of degree $n$. A number $c\in\C$ is a root of
polynomial $\chi_V^{(\infty)}(\al)$ if and only if there exists a
polynomial $p(x,y)\in\C[x,y]$ such that $\deg_xp(x,y)=s$ and
$x^{c-s}p(x,\log x)\in V$. A number $c\in\C$ is a root of polynomial
$\chi_V^{(0)}(\al)$ if and only if there exists a polynomial
$p(x,y)\in\C[x,y]$ such that $p(0,y)\neq 0$ and $x^c p(x,\log x)\in
V$.

\begin{lem}
There exists a unique monic polynomial $Y_V(x)$ such that
\be
\frac{\chi_V^{(\infty)}(\al)}{\chi_V^{(0)}(\al)}=\frac{Y_V(\al-1)}{Y_V(\al)}.
\ee
\end{lem}
\begin{proof}
The lemma is obvious.
\end{proof}

The Wronskian of any
basis for $V$ has the form $w(x)x^r$, where $r\in\C$ and
$w(x)\in\C[x]$, $w(0)\neq 0$. The unique representative with a monic
polynomial $w(x)$ is called the {\it Wronskian of $V$} and is denoted
by $\Wr(V)$.

\begin{theorem}\label{theorem 3} Let $V$ be an unramified
space of quasi-polynomials
with real exponents $\bs z\in\R^n$, $Y_V=\prod_{i=1}^m
(x-\tilde z_i)$ and $\Wr(V)=x^r\prod_{i=1}^N(x-Q_i)$, where
$\prod_{i=1}^NQ_i\neq 0$.
Assume that $Q_1,\dots,Q_N$ are real. We have:
\begin{enumerate}
\item If $|\tilde z_i-\tilde z_j|\geq 1$ for all $i\neq j$,
then the space $V$ is real.

\item Let $Q_1,\dots,Q_N$ be either all positive or all negative.
Assume that there exists a subset $I\subset\{1,\dots,n\}$ such that
$|\tilde z_i-\tilde z_j|\geq 1$ for $i\neq j$ provided either $i,j\in I$ or
$i,j\not\in I$.
Then the space $V$ is real.
\end{enumerate}
\end{theorem}

Theorem \ref{theorem 3} is proved in Section \ref{proof sec 3}.

\subsection{Proof of Theorem \ref{theorem 3}}\label{proof sec 3}
If $x^z\in V$ for some $z\in\R$ then $\tilde V=(x\partial-z)V$ is an
unramified space of quasi-polynomials of dimension $n-1$,
with the same exponents (except maybe for $z$).
We have $\Wr(V)=\Wr(\tilde V)$ and
$Y_{\tilde V}=Y_V$. Moreover, if
$\tilde V$ is real then $V$ is real. Therefore, without loss of
generality we can assume that $V$ is non-degenerate.

Let $V$ be an unramified non-degenerate space of quasi-polynomials
with real exponents $\bs z\in\R^n$, $Y_V=\prod_{i=1}^m
(x-\tilde z_i)$ and $\Wr(V)=x^r\prod_{i=1}^N(x-Q_i)$, where $r\in\C$,
$Q_i\neq 0$. Let
\be
\mc D_V=(x\partial)^n+A_1(x)(x\partial)^{n-1}+\dots +A_n(x)
\ee
be the unique differential operator of order $n$ with kernel $V$
and the top coefficient $x^n$.
The coefficients $A_i(x)$ are rational functions in $x$.
Let $\bar A_0(x)\in\C[x]$
be a monic polynomial such that
$\bar A_i(x)=A_i(x)\bar A_0(x)$ is a polynomial for
$i=1,\dots,n$, and polynomials $\bar A_0(x),\dots, \bar A_n(x)$
are relatively prime.

Write
\be
A_0(x)\mc D_V=\sum_{i=1}^s\sum_{j=1}^n\bar A_{ij}x^i(x\partial)^j,
\ee
where $\bar A_{ij}\in\C$ and $s=\max_i(\deg \bar A_i(x))$.
It is sufficient to prove that $\bar A_{ij}$ are real numbers.

Define a difference operator with polynomial coefficients
$\mc D_V^*$ by the formula
\be
\mc D_V^*=\sum_{i=1}^s\sum_{j=1}^n\bar A_{ij}x^je^{-i\partial}.
\ee

\begin{prop}\label{cite prop}
The quasi-exponential kernel of the operator $\mc D_V^*$ has the
form $V^*\otimes \mc U$, where $V^*$ is a space of
quasi-exponentials with bases $(\bar Q_1,\dots,\bar Q_s)$ and $\bar
Q_i\in\{Q_1,\dots,Q_N\}$ for $i=1,\dots,s$. Moreover,
$\Wr^d(V^*)=Y_V$.
\end{prop}
\begin{proof}
Proposition \ref{cite prop} is proved using a suitable integral
transform in the same way as Theorem 4.1 in \cite{MTV5}. An
alternative proof can be found in \cite{BC}.
\end{proof}

By Theorem \ref{theorem 1}, the space $V^*$ is real and therefore
all $\bar A_{ij}$ are real numbers.

\section{Reformulations}\label{app}

\subsection{The discrete case}\label{dm sec}
In this section we give a reformulation of
Theorem \ref{theorem 1}.

Fix $\bs Q=(Q_1,\dots,Q_N)\in(
\C^\times)^N$, such that $Q_i\neq Q_j$ if
$i\neq j$.

Let $S$ be the $N\times N$
Vandermonde matrix with the $(i,j)$ entry
\be
s_{ij}=Q_i^{j-1}.
\ee
We have $\det S=\prod_{i<j}(Q_j-Q_i)$.

Let $\bar S$ be the $N\times N$
matrix with the $(i,j)$ entry
\be
\bar s_{ij}=(j-1)Q_i^{j-1}.
\ee

Let $A=\diag(a_1,\dots,a_N)$ be the diagonal matrix
with diagonal entries $a_1,\dots,a_N$.

Clearly, the discrete Wronskian of $N$
quasi-exponentials with linear polynomial
part is given by
\bean\label{simple wr}
\Wr^d_1((x-a_1)Q_1^{x},\dots, (x-a_N)Q_N^x)= (\prod_{i=1}^N{Q_i^x})
\det\left((x-A)S+\bar S\right).
\eean

Denote $\bar S S^{-1}=M$. Let $m_{ij}$ denote the $(i,j)$ entry of $M$.

\begin{lem} We have
\bea
m_{ij}&=&Q_i\frac{\prod_{s\neq i,j}(Q_i-Q_s)}{\prod_{s\neq j}(Q_j-Q_s)}
\qquad
(i\neq j),\\
m_{ii}&=& Q_i\sum_{s\neq i}\frac{1}{Q_i-Q_s}.
\eea
\end{lem}
\begin{proof}
Let $s^*_{ij}$ denote the $i,j$ entry of $W^{-1}$.
Define the polynomials $s^*_i(u)=\sum_{s=1}^Ns^*_{js}u^{s-1}$. We have
\be
s^*_j(Q_i)=\delta_{ij}.
\ee
Therefore
\be
s^*_j(u)=\frac{\prod_{s\neq j} (u-Q_s)}{\prod_{s\neq j} (Q_j-Q_s)}.
\ee
Further, we have
\be
m_{ij}=Q_i(\frac{d}{du}\ s^*_j)|_{u=Q_i}.
\ee
The lemma follows.
\end{proof}

Define the matrix $\mc Z^d$ by
\bean\label{zd}
\mc Z^d =
\left( \begin{array} {cccccc}
a_1 & \dfrac{Q_1}{Q_2-Q_1} & \dfrac{Q_1}{Q_3-Q_1} &\dots & \dfrac{Q_1}{Q_N-Q_1}
\\
\dfrac{Q_2}{Q_1-Q_2} & a_2 & \dfrac{Q_2}{Q_3-Q_2} &{} \dots &
\dfrac{Q_2}{Q_N-Q_2}
\\[9pt]
{}\dots & {}\dots & {} \dots & \dots & \dots
\\[6pt]
\dfrac{Q_N}{Q_1-Q_N} & \dfrac{Q_N}{Q_2-Q_N}& \dfrac{Q_N}{Q_3-Q_N} &{} \dots
& a_N
\end{array} \right).
\eean

Let $D=\diag(d_1,\dots,d_N)$ be the diagonal matrix with diagonal entries
\be
d_i=\prod_{s\ne i} (Q_i-Q_s).
\ee
Let $B=\diag(m_{11},\dots,m_{NN})$ be the diagonal matrix with diagonal
entries $m_{ii}$.

Then we have the equality of matrices
\bean\label{conjugate}
A+B-D^{-1}\bar S S^{-1} D=\mc Z^d.
\eean

\begin{theorem}\label{theorem 1a}
Let
$Q_1,\dots,Q_N$ be distinct real numbers and $a_1,\dots,a_N$ complex numbers.
Let $z_1,\dots,z_N$ be eigenvalues of the matrix $\mc Z^d$.
Assume that $z_1,\dots,z_N$ are real.
Then we have:
\begin{enumerate}
\item If $|z_i-z_j|\geq 1$ for all $i\neq j$,
then the numbers $a_1,\dots,a_N$ are real.

\item Let $Q_i$ be either all positive or all negative.
Assume that there exists a subset $I\subset\{1,\dots,N\}$ such that
$|z_i-z_j|\geq 1$ for $i\neq j$ provided either $i,j\in I$ or $i,j\not\in I$.
Then the numbers $a_1,\dots,a_N$ are real.
\end{enumerate}
\end{theorem}
\begin{proof}
By formulas \Ref{simple wr}, \Ref{conjugate} the eigenvalues of the
matrix $\mc Z^d$ are zeroes of the discrete Wronskian $\Wr^d((x-\bar
a_1)Q_1^{x},\dots, (x-\bar a_N)Q_N^x)$, where $\bar a_i=a_i+m_{ii}$.
Since $m_{ii}$ are real, Theorem \ref{theorem 1a} follows from
Theorem \ref{theorem 1}.
\end{proof}

\begin{rem}
We are not aware of a direct proof of Theorem \ref{theorem 1a}.
\end{rem}

\begin{cor}\label{cor 1}
Let $Q$ and $Z$ be complex $N\times N$-matrices such that $Q$ is
non-singular and
\be
Z - Q^{-1}ZQ = 1 - K,
\ee
where $K$ is a rank one matrix.
Assume that all eigenvalues of $Q$ and $Z$ are real.
Let $z_1,\dots,z_N$ be the eigenvalues of $Z$.
Then we have:
\begin{enumerate}
\item
If $|z_i-z_j|\geq 1$ for $i\neq j$,
then there exists an invertible matrix $C$ such that $C^{-1}QC$ and
$C^{-1}ZC$ are real matrices.

\item
Let eigenvalues of $Q$ be either all positive or all negative.
Assume that there exists a subset $I\subset\{1,\dots,N\}$ such that
$|z_i-z_j|\geq 1$ for $i\neq j$ provided either $i,j\in I$ or $i,j\not\in I$.
Then there exists an invertible matrix $C$ such that $C^{-1}QC$ and
$C^{-1}ZC$ are real matrices.
\end{enumerate}
\end{cor}
\begin{proof}
Let $\tilde {\mc M}_N$ be the set of pairs of complex $N\times N$
matrices $(Z,Q)$ such that $Q$ is non-singular and such that the rank
of the matrix $Z-Q^{-1}ZQ-1$ is one.
We call $(Z_1,Q_1),(Z_2,Q_2)\in \tilde {\mc M}_N$
equivalent if there exists an invertible $N\times N$
matrix $C$ such that $Z_2=C^{-1}Z_1C$ and $Q_2=C^{-1}Q_1C$.
Let ${\mc M_N}$ be the set of equivalence classes.

Define a map:
\be
\tau: \ \mc M_N \to \C^N/S_N\times \C^N/S_N,
\ee
which sends the class of $(Z,Q)$ to $(\on{Spec}\,Z\ ,\on{Spec}\, Q)$.

Then similarly to \cite{EtG},
one can show that $\mc M_N$ is a smooth variety and the
map $\tau$ is a finite map of degree $N!$, \cite{Et}.

Let $\R\mc M_N\subset \mc M_N$ be the subset of classes of pairs
$(Z,Q)\in\mc M_N$ with real matrices $Z,Q$.
By Proposition 3.1 of \cite{HY}, the subset $\R\mc M_N$ is closed.

If $Q$ is a semi-simple matrix with eigenvalues $Q_i$, then there
exists a matrix $C$ such that $C^{-1}QC$ is diagonal and then it is
easy to see that $Z$ is given by \Ref{zd}.

Therefore, the corollary follows from Theorem~\ref{theorem 1a} by continuity.
\end{proof}

\subsection{The smooth case}
In this section we give a reformulation of
Theorem \ref{theorem 2}.

Fix $\bs \la=(\la_1,\dots,\la_N)$, such that $\la_i\neq \la_j$ if $i\neq j$.
Define the matrix $\mc Z$ by
\bean\label{z}
\mc Z =
\left( \begin{array} {cccccc}
a_1 & \dfrac{1}{\la_2-\la_1} & \dfrac{1}{\la_3-\la_1} &\dots &
\dfrac{1}{\la_N-\la_1}
\\
\dfrac{1}{\la_1-\la_2} & a_2 & \dfrac{1}{\la_3-\la_2} &{} \dots &
\dfrac{1}{\la_N-\la_2}
\\[9pt]
{}\dots & {}\dots & {} \dots & \dots & \dots
\\[6pt]
\dfrac{1}{\la_1-\la_N} & \dfrac{1}{\la_2-\la_N}& \dfrac{1}{\la_3-\la_N}
&{} \dots & a_N
\end{array} \right).
\eean

\begin{theorem}\label{theorem 2a}
Let $\la_1,\dots,\la_N$ be distinct real numbers. If all
eigenvalues of the matrix $\mc Z$ are real, then the numbers
$a_1,\dots,a_N$ are real.
\end{theorem}
\begin{proof}
By a computation similar to the one in Section \ref{dm sec},
we obtain that the
eigenvalues of the matrix $\mc Z$ are zeroes of the Wronskian
$\Wr((x-\tilde
a_1)e^{\la_1x},\dots, (x-\tilde a_N)e^{\la_Nx})$, where
\be
\tilde a_i=a_i+\sum_{s\neq i}\frac{1}{\la_i-\la_s}.
\ee
Therefore Theorem \ref{theorem 2a} follows from Theorem \ref{theorem 2}.
\end{proof}

The matrix $\mc Z$ in relation to the Wronskian of quasi-exponentials
appeared in \cite{W}.

\begin{cor}\label{cor 2} \cite{HY}
Let $Q$ and $Z$ be complex $N\times N$-matrices such that
\be
[Q,Z] = 1 - K,
\ee
where $K$ is a rank one matrix.
If all eigenvalues of $Q$ and $Z$ are real,
then there exists an invertible matrix $C$ such that $C^{-1}QC$ and
$C^{-1}ZC$ are real matrices.
\end{cor}
\begin{proof}
The proof is similar to the proof of Corollary \ref{cor 1}.
\end{proof}

We are not aware of a direct proof of Theorem \ref{theorem 2a}.

Using the duality studied in \cite{MTV4}, one can show that the case
of quasi-exponentials with linear polynomials is generic, so
Theorem~\ref{theorem 2} can be deduced from Theorem~\ref{theorem 2a}.

Moreover, the B.~and M.\,Shapiro conjecture can be obtained from
Theorem~\ref{theorem 2a} for the case of a nilpotent matrix $\mc Z$.

For a different proof of Corollary \ref{cor 2},
a relation to Calogero-Moser spaces and representations of the
Cherednik algebras see \cite{GHY}.

\subsection{Trigonometric case}
In this section we give a dual version of Theorem \ref{theorem 1a}.

Fix complex numbers $z_1,\dots,z_N$ such that $z_i-z_j\neq 1$.

Define the matrix $\mc Q^d$ by
\be
\mc Q^d =
\left( \begin{array} {cccccc}
b_1 & \dfrac{b_2}{z_1-z_2+1} & \dfrac{b_3}{z_1-z_3+1}
&\dots & \dfrac{b_N}{z_1-z_N+1}
\\
\dfrac{b_1}{z_2-z_1+1} & b_2 & \dfrac{b_3}{z_2-z_3+1} &{} \dots &
\dfrac{b_N}{z_2-z_N+1}
\\
{}\dots & {}\dots & {} \dots & \dots & \dots
\\
\dfrac{b_1}{z_N-z_1+1} & \dfrac{b_2}{z_N-z_2+1}& \dfrac{b_3}{z_N-z_3+1}
&{} \dots & b_N
\end{array} \right).
\ee

\begin{theorem}\label{theorem 3a} Let
$z_1,\dots,z_N$ be real numbers such that $z_i-z_j\neq 1$
and let $b_1,\dots,b_N$ be complex numbers.
Let $Q_1,\dots,Q_N$ be eigenvalues of the matrix $\mc Q^d$. Assume
that $Q_1,\dots, Q_N$ are nonzero distinct real numbers. Then we have:
\begin{enumerate}
\item If $|z_i-z_j|>1$ for all $i\neq j$,
then the numbers $b_1,\dots,b_N$ are real.

\item Let $Q_i$ be either all positive or all negative.
Assume that there exists a subset $I\subset\{1,\dots,N\}$ such that
$|z_i-z_j|\geq 1$ for $i\neq j$ provided either $i,j\in I$ or $i,j\not\in I$.
Then the numbers $b_1,\dots,b_N$ are real.
\end{enumerate}
\end{theorem}

\begin{proof}
Let $Z=\diag(z_1,\dots,z_N)$ be the diagonal matrix with diagonal
entries $z_i$. Then $QZ-ZQ-Q$ is a rank $1$ matrix. Theorem
\ref{theorem 3a} follows from Corollary \ref{cor 1}.

Alternatively,
one can show that $z_1,\dots,z_N$ are all distinct, that
the eigenvalues of the matrix $\mc Q^d\prod_{i<j}$
are zeroes of the Wronskian
$\Wr(x^{z_1}(x-\tilde b_1),\dots,x^{z_N}(x-\tilde b_N))$,
where
\be
\tilde b_i=b_i\prod_{s\neq i}\frac{z_i-z_s}{z_i-z_s-1},
\ee
and deduce Theorem \ref{theorem 3a} from Theorem \ref{theorem 3}.
\end{proof}

\end{document}